\documentclass[10pt]{article}

\usepackage[english]{babel}

\usepackage{amsfonts}
\usepackage{amssymb}
\usepackage{amsthm}

\newtheorem{theo}{Theorem}[section]
\newtheorem{prop}{Proposition}[section]
\newtheorem{lem}{Lemma}[section]
\newtheorem{cor}{Corollary}[section]
\newtheorem{DEF}{Definition}[section]
\newtheorem{EX}{Example}[section]
\newtheorem{REM}{Remark}[section]
\newenvironment{theorem}{\begin{theo}}{\end{theo}}  
\newenvironment{proposition}{\begin{prop}}{\end{prop}}

\newenvironment{definition}{\begin{DEF}\rm}{\end{DEF}} 
\newenvironment{example}{\begin{EX}\rm}{\end{EX}} 
\newenvironment{remark}{\begin{REM}\rm}{\end{REM}}

\newcommand{\R}{\mathbb R}

\newcommand{\N}{\mathbb N}
\newcommand{\B}{\mathbb B}

\newcommand{\Hi}{\mathbb H}

\newcommand{\Sph}{\mathbb S}
\newcommand{\W}{\mathbb W}
\newcommand{\X}{\mathbb X}
\newcommand{\Y}{\mathbb Y}

\newcommand{\cl}{{\rm cl\,}}
\newcommand{\bd}{{\rm bd\,}}

\newcommand{\inte}{{\rm int\,}}

\newcommand{\Lagr}{{\rm L}}
\newcommand{\CII}{{\rm C}^{1,1}}
\newcommand{\FDer}{{\rm D}}
\newcommand{\Lin}{\mathcal{L}}
\newcommand{\krn}{{\rm ker\, }}
\newcommand{\nullv}{{\bf 0}}
\newcommand{\aie}{\mbox{\boldmath $\omega$}}

\newcommand{\Vmax}[1]{{#1}\hbox{-}\max}
\newcommand{\commv}[1]{{\bf {#1}}}
\newcommand{\modconv}[1]{\delta_{#1}}

\newcommand{\ball}[2]{{\rm B}({#1};{#2})}
\newcommand{\NCone}[2]{{\rm N}\left({#1};{#2}\right)}
\newcommand{\Image}[2]{\mathcal{I}_{#1}^{{}_{#2}}}

\title{\bf Localizing Vector Optimization Problems \\ with Application to Welfare Economics}

\author{A. Uderzo\footnote{{\sc Department of Mathematics and Applications,
University of Milano-Bicocca}, Via Cozzi, 53 - 20125 Milano, Italy, \hskip3cm
{\it e-mail address}: {\tt amos.uderzo@unimib.it}}}

\begin{document}

\maketitle

\vskip3cm

\begin{abstract}
In the present paper, the Polyak's principle, concerning
convexity of the images of small balls through $\CII$ mappings,
is employed in the study of vector optimization problems. This
leads to extend to such a context achievements of local programming,
an approach to nonlinear optimization, due to B.T. Polyak,
which consists in exploiting the benefits of the convex local
behaviour of certain nonconvex problems.
In doing so, solution existence and optimality conditions are
established for localizations of vector optimization problems,
whose data satisfy proper assumptions. Such results are subsequently
applied in the analysis of welfare economics, in the case of an
exchange economy model with infinite-dimensional commodity
space. In such a setting, the localization of an economy yields
existence of Pareto optimal allocations, which, under certain additional
assumptions, lead to competitive equilibria.
\end{abstract} 

\vskip3cm

\noindent{\bf Mathematics Subject Classification (2010):}
Primary: 58E17; Secondary: 47N10, 90C29, 90C48, 91B15.

\vskip.5cm

\noindent{\bf Key words:} modulus of convexity; Polyak's convexity
principle; openness at a linear rate; Lagrangian function; vector
optimization; $\epsilon$-localization of a problem; exchange
economy; regular feasible allocation; Pareto optimality; competitive
equilibrium.

\vfill


\section{Introduction}

In a series of papers appeared at the beginnings of the current
millennium (see \cite{Poly01a,Poly01b,Poly03}), the term ``local
programming" was used to denote the theory emerging in connection
with a special class of nonlinear optimization problems. This class
includes mathematical programming problems, with equality and inequality
constraints, that, even in the absence of convexity assumptions on
their data, surprisingly do exhibit a local behaviour, which is typical
of convex optimization problems. The doubtless advantages arising
when one handles problems with convex data should underline the
importance of local programming. The unexpected appearance of
a local convex behaviour within the ``ocean of nonlinear optimization"
has a deep reason, resting upon the Polyak's convexity principle.
This crucial achievement of modern nonlinear analysis states,
in its original formulation, that a mapping between Hilbert spaces,
which is $\CII$ around a regular point, carries balls centered at
that point to convex sets, provided that the radius of the balls is
small enough. In many questions related to optimization, it is
already the convexity of images of sets, not only that of the involved
functions, which does the trick, with a lot of proficuous consequences.
Then reason why $\CII$ smoothness of a mapping along with its regularity
should imply convexity of the image of balls is even deeper, having
to do with profound geometric properties of the underlying space
and with the preservation of convexity through linear approximations.

In the present paper, such a ultimate reason is left at that.
Instead, the main theme is the extension of the local programming
approach to vector optimization. In fact, also in such context,
a class of nonlinear problems can be singled out, whose local
convex behaviour bear interesting consequences. The study of
them is carried out in a general setting. Constrained vector
optimization problems will be supposed to be defined in a proper
subclass of reflexive Banach spaces. Nonetheless, some of the
findings that are going to be presented here seem to be novel
even for finite-dimensional problems.

The material exposed in the paper is organized as follows. In
Section \ref{Sect:2}, key concepts and results from nonlinear
analysis, essentially employed in subsequent investigations,
are recalled, along with the most part of the notation
in use throughout the paper. Section \ref{Sect:3}
contains the main result of the paper, describing the effect of
localizing problems in nonlinear vector optimization. It deals,
in particular, with existence of solutions and optimality conditions
for detecting them. Section
\ref{Sect:4} is reserved for an application of the main result
to a topic of mathematical economics, known as welfare theory.
More precisely, a model of (pure) exchange economy, with an
infinite-dimensional commodity space and finitely many consumers,
is considered. In such model, the existence of Pareto optimal allocations,
which, under an adequate qualification, turn out to be also equilibria,
for proper localizations of the original economy is obtained.


\section{Mathematical preliminaries}     \label{Sect:2}

Throughout the paper, whenever $(\X,\|\cdot\|)$ denotes a Banach space,
$\ball{x}{r}$ denotes the ball with centre at $x\in\X$ and radius
$r\ge 0$. The null vector of a Banach space is marked by $\nullv$.
The unit ball, i.e. the set $\ball{\nullv}{1}$, is simply denoted by
$\B$, whereas the unit sphere by $\Sph$. If $S$ is a subset of
a Banach space, $\inte S$, $\bd S$ and $\cl S$ denote the interior,
the boundary and the (topological) closure of $S$, respectively.
Fixed $x\in S$, $\NCone{S}{x}$ denotes the normal cone to $S$
at $x$ in the sense of convex analysis.

For the purposes of the present analysis, general Banach spaces
are a setting too wide. In fact, the main result presented in this
paper and its application essentially rely on certain geometrical
features of a specific class of Banach spaces, features that are
related to the rotundity of the balls. The rotundity property of a
Banach space $(\X,\|\cdot\|)$ can be quantitatively described
by means of the function $\modconv{\X}: [0,2]\longrightarrow [0,1]$,
defined by
$$
    \modconv{\X}(\epsilon)=\inf\left\{1-\left\|\frac{x_1+x_2}{2}\right\|:\ 
    x_1,\, x_2\in\B,\ \|x_1-x_2\|\ge\epsilon\right\},
$$
which is called the {\em modulus of convexity} of $(\X,\|\cdot\|)$. 
Notice that $\modconv{\X}$ is not invariant under equivalent
renormings of $\X$. Such notion allows one to define a special
class of Banach spaces, whose introduction is due to J.A.
Clarkson (see, for instance, \cite{Clar36,FaHaMoPeZi01,Megg98}). 

\begin{definition}
A Banach space  $(\X,\|\cdot\|)$ is called {\em uniformly convex}
(or, {\em uniformly rotund}) if
it is $\modconv{\X}(\epsilon)>0$ for every $\epsilon\in (0,2]$.
\end{definition}

In what follows, the modulus of convexity of a (uniformly convex)
Banach space is said to fulfil the {\em quadratic growth condition}
if there exists $\kappa>0$ such that
$$
    \modconv{\X}(\epsilon)\ge\kappa\epsilon^2,\quad\forall
   \epsilon\in [0,2].
$$
The class of uniformly convex Banach spaces, with modulus of
convexity fulfilling the quadratic growth condition, reveals to be
the proper setting, in which to develp the analysis of the question
under consideration.

\begin{example}
(i) Since the modulus of convexity of a Hilbert space $\Hi$ can be
easily calculated to be
$$
    \modconv{\Hi}(\epsilon)=1-\sqrt{1-\frac{\epsilon^2}{4}},\quad\forall
    \epsilon\in [0,2],
$$
it is clear that every Hilbert space is uniformly convex, with a
modulus of convexity fulfilling the quadratic growth condition
with $0<\kappa\le 1/8$.

(ii) More generally, such Banach spaces as  $l^p$, $L^p$, and $W^p_m$,
with $1<p<2$, are known to have a modulus of convexity satisfying
the relation
$$
    \modconv{l^p}(\epsilon)=\modconv{L^p}(\epsilon)=\modconv{W^p_m}
    (\epsilon)> \frac{p-1}{8}\epsilon^2,\quad\forall\epsilon\in (0,2].
$$
Therefore, they also are an example of uniformly convex space with
a modulus of convexity satisfying the quadratic growth condition
(see, for instance, \cite{FaHaMoPeZi01}).

(iii) Concerning the notion of uniform convexity, a caveat is due:
even finite-dimensional Banach spaces may fail to be uniformly
convex. Consider, for instance, $\R^2$ equipped with the Banach space
structure given by the norm $\|\cdot\|_\infty$.
\end{example}

\begin{remark}     \label{rem:UCspace}
(i) In \cite{Nord60} it was proved that the modulus of convexity
$\modconv{\X}$ of every real Banach space, having dimension greater
than  $1$, admits the following estimate from above
$$
    \modconv{\X}(\epsilon)\le1-\sqrt{1-\frac{\epsilon^2}{4}},\quad\forall
   \epsilon\in [0,2].
$$
This implies that the quadratic growth is a maximal one.

(ii) In the sequel, the fact will be used that every uniformly convex
Banach space is reflexive. In the Banach space theory, such result
is known under the name of Milman-Pettis theorem (see \cite{Megg98}).
It is worth mentioning that uniform convexity is not characterized
by reflexivity. Indeed, in \cite{Day41} a large class of reflexive 
(separable and strictly convex) Banach spaces is exhibited,
which are not isomorphic to uniformly convex spaces.

(iii) Let $(\X,\|\cdot\|)$ be a uniformly convex Banach space, with a modulus
of convexity $\modconv{\X}$ fulfilling the quadratic growth condition,
and $n\in\N$.  As a consequence of Theorem 5.2.25 in \cite{Megg98},
the space $(\X^n,\|\cdot\|_2)$, where the direct sum space
$\X^n=\X\oplus\dots\oplus\X$ is normed with the $2$-norm
$\|\cdot\|_2:\X^n\longrightarrow [0,+\infty)$
$$
    \|(x_1,\dots,x_n)\|_2=\left(\sum_{i=1}^n\|x_i\|^2\right)^{1/2},
$$
is a uniformly convex space.
\end{remark}


Another key concept, playing a crucial role in this paper, is openness
at a linear rate, a property for mappings, which postulates a certain
quantitative surjectivity behaviour. More precisely, 
a mapping $f:\X\longrightarrow\Y$ between Banach
spaces is said to be {\em open at a linear rate} around $(x_0,f(x_0))$,
with $x_0\in\X$, if there exist positive $\delta$, $\zeta$ and
$\sigma$ such that
\begin{equation}      \label{in:loclinopen}
    f(\ball{x}{r})\supseteq\ball{f(x)}{\sigma r}\cap\ball{f(x_0)}{\zeta},
      \quad\forall x\in\ball{x_0}{\delta},\ \forall r\in [0,\delta).
\end{equation}
Clearly, inclusion $(\ref{in:loclinopen})$ has crucial consequences
on the local solvability of the equation $f(x)=y$ as well as on the
Lipschitz behaviour of its solution set $f^{-1}(y)$ near $x_0$. So, it comes
not surprising that many efforts have been directed to find out
criteria able to detect the occurence of such a property.
The following result, known in nonlinear analsysis
as Lyusternik-Graves theorem,
provides a characterization of openness at a linear rate
for strictly differentiable mappings (see, for instance,
Theorem 1.57 in \cite{Mord06I}). Throughout the paper, the Fr\'echet
derivative at $x\in\X$ of a mapping $f:\X\longrightarrow\Y$ between
Banach spaces is denoted by $\FDer f(x)$.

\begin{theorem}     \label{thm:LyuGra}
Let $f:\X\longrightarrow\Y$ be a mapping between Banach spaces.
Suppose $f$ to be strictly differentiable at $x_0\in\X$. Then
$f$ is open at a linear rate around $(x_0,f(x_0))$ iff
$\FDer f(x_0)$ is onto.
\end{theorem}

For the analysis conducted in the present paper, strict differentiability
will be not enough. Instead, in the main result, mappings will be
supposed to be $\CII$. Recall that a mapping $f:\X\longrightarrow\Y$
between Banach spaces is said to be $\CII(\Omega)$, with $\Omega$
being an open subset of $\X$, if it admits Fr\'echet derivative at $x$,
for every $x\in\Omega$, and the mapping $\FDer f:\Omega
\longrightarrow\Lin(\X,\Y)$ is Lipschitz continuous on $\Omega$,
where $\Lin(\X,\Y)$ stands for the Banach space of all linear bounded
mappings between $\X$ and $\Y$, equipped with the operator norm.
In the special case $\Y=\R$, the symbol $\Lin(\X,\Y)$ is replaced
by $\X^*$. Given $x^*\in\X^*$, its kernel is denoted by $\krn x^*$.

The main tool of analysis in the subsequent section will be the
Polyak's convexity principle. It states that $\CII$ mappings, which
are open at a linear rate around a given point, carry small balls
centered at that point to convex sets. This important result was
originally established for mappings between Hilbert spaces (see
\cite{Poly01a,Poly01b,Poly03}) and, later on, it was extended to mappings
defined on uniformly convex Banach spaces, with modulus of
convexity fulfilling the quadratic growth condition (see \cite{Uder13}).
In order to give the present analysis a proper level of generality,
motivated by applications to models of welfare economics exposed
in the last section, the Polyak's convexity principle is below
formulated in its most recent version.

\begin{theorem}    \label{thm:Polyakteo}
Let $f:\X\longrightarrow\Y$ be a mapping between Banach spaces, let $\Omega$
be an open subset of $\X$, let $x_0\in\Omega$, and $r>0$ such that $\ball{x_0}{r}
\subseteq\Omega$. Suppose that:

\noindent $(i)$ $(\X,\|\cdot\|)$ is uniformly convex with modulus $\delta_\X$
satisfying the quadratic growth condition;

\noindent $(ii)$ $f\in\CII(\Omega)$ and $\FDer f(x_0)\in\Lin(\X,\Y)$
is onto.

\noindent Then, there exists $\epsilon_0\in (0,r)$ such that
$f(\ball{x_0}{\epsilon})$ is convex, for every $\epsilon\in
[0,\epsilon_0]$.
\end{theorem}

\begin{remark}     \label{rem:bdPolyprin}
The following complement of Theorem \ref{thm:Polyakteo},
already remarked in \cite{Poly01a}, will be exploited in the sequel. 
From hypothesis (ii) and Theorem \ref{thm:LyuGra}, one has that
$f(\inte \ball{x_0}{\epsilon}) \subseteq\inte f(\ball{x_0}{\epsilon})
\ne\varnothing$, for every $\epsilon\in(0,\epsilon_0]$. Therefore,
it holds
$$
    f^{-1}(\bd f(\ball{x_0}{\epsilon}))\subseteq
    \bd\ball{x_0}{\epsilon}.
$$
\end{remark}


\section{A localization property in vector optimization}   \label{Sect:3}

Consider a vector optimization problem of the following form:
$$
   \Vmax{K}_{x\in\X} h(x) \quad\hbox{subject to}\quad g(x)\in C,
   \leqno (\mathcal{VOP})
$$
where $h:\X\longrightarrow\W$ and $g:\X\longrightarrow\Y$ are
given mappings, $C$ is a nonempty subset of $\Y$ and $K\subseteq\W$
is a closed, convex and pointed cone (with apex at the null vector of $\W$).
Here $(\X,\|\cdot\|)$, $(\W,\|\cdot\|)$ and $(\Y,\|\cdot\|)$ are real
Banach spaces. Besides, $\W$ is supposed to be partially ordered by $K$
in the canonical way, namely $K$ induces a partial order relation 
$\le_K$ over elements of $\W$ as follows
$$
    w_1\le_K w_2 \qquad\hbox{iff}\qquad w_2-w_1\in K.
$$
In other terms, $K$ can be regarded as the positive cone with
respect to a partial ordering $\le_K$ defined on $\W$. By $K^\oplus
=\{w^*\in\W^*:\ \langle w^*,w\rangle\ge 0,\quad\forall w\in K\}$ the
positive dual cone of $K$ is denoted.
The feasible region associated with $(\mathcal{VOP})$ is indicated by
$$
   R=\{x\in\X:\ g(x)\in C\}=g^{-1}(C).
$$
Recall that $\bar x\in R$ is said to be  {\em locally $K$-optimal} for (or a
{\em local solution} to) $(\mathcal{VOP})$ if there exists $r>0$ such that
$$
    h(R\cap\ball{\bar x}{r})\cap(h(\bar x)+K)=\{h(\bar x)\}.
$$
Of course, if in the above equality $\ball{\bar x}{r}$ can be
replaced by $\X$, $\bar x$ is also {\em globally $K$-optimal}
for $(\mathcal{VOP})$.

According to a longstanding approach, in order to investigate
optimization problems of the form $(\mathcal{VOP})$, given an
element $\bar x\in R$ it is convenient to associate with such problem
the mapping $\Image{\bar x}{\mathcal{VOP}}
:\X\longrightarrow\W\times\Y$, defined by
\begin{eqnarray}    \label{eq:defimage}
    \Image{\bar x}{\mathcal{VOP}}(x)=(h(x)-h(\bar x),g(x)).
\end{eqnarray}
By means of such mapping, letting
$$
    \mathcal{Q}=(K\backslash\{\nullv\})\times C,
$$
one is in a position to formulate the following set characterization
of local $K$-optimality.

\begin{proposition}    \label{pro:Koptchar}
An element $\bar x\in R$ is a local solution to $(\mathcal{VOP})$
iff there exists $r>0$ such that
$$
   \Image{\bar x}{\mathcal{VOP}}(\ball{\bar x}{r})\cap\mathcal{Q}
   =\varnothing.
$$
\end{proposition}

\begin{proof}
The proof stems directly from the definition of $\Image{\bar x}
{\mathcal{VOP}}$ and from the aforementioned notion of local
$K$-optimality. 
\end{proof}

Within the context of vector optimization,
the issue addressed in this section deals with the local behaviour
of $(\mathcal{VOP})$ near certain reference points of its feasible
region. The approach here proposed leads to introduce the concept
of problem localization.
Let $x_0\in R$ and $\epsilon>0$. By {\em $\epsilon$-localization}
of $(\mathcal{VOP})$ around $x_0$ the following problem
is meant
$$
   \Vmax{K}_{x\in\ball{x_0}{\epsilon}} h(x) \quad\hbox{subject to}
    \quad g(x)\in C\leqno (\mathcal{VOP}_{x_0,\epsilon})
$$
The reader should notice that,  because $\ball{x_0}{\epsilon}$ is closed,
$(\mathcal{VOP}_{x_0,\epsilon})$ actually contains a further constraint.
Its introduction may change substantially the geometry of the problem.

In order to investigate the effect of localizing vector optimization,
the next general proposition is needed, which
shows how openness at a linear rate of
$\Image{x_0}{\mathcal{VOP}}$ can not be consistent with the
$K$-optimality of a feasible element $x_0\in R$.

\begin{proposition}      \label{pro:openvectopt}
With reference to a problem  $(\mathcal{VOP})$, let $x_0\in R$.
If mapping $\Image{x_0}{\mathcal{VOP}}$ is open at a linear
rate near $(x_0,(\nullv,g(x_0)))$, then $x_0$ fails to be a solution to
$(\mathcal{VOP}_{x_0,\epsilon})$, for every $\epsilon>0$.
\end{proposition}

\begin{proof}
Fix an arbitrary $\epsilon>0$. By hypothesis, according to $(\ref{in:loclinopen})$
there exist positive $\delta$, $\zeta$ and $\sigma$ such that
\begin{equation}
    \Image{x_0}{\mathcal{VOP}}(\ball{x}{r})\supseteq
   \ball{\Image{x_0}{\mathcal{VOP}}(x)}{\sigma r}\cap
   \ball{(\nullv,g(x_0))}{\zeta},  \quad\forall 
  x\in\ball{x_0}{\delta},\ \forall r\in [0,\delta).
\end{equation}
Thus, by taking $x=x_0$ and $r$ such that
$$
    0<r<\min\left\{\epsilon,\frac{\zeta}{\sigma}\right\},
$$
one finds
$$
   \Image{x_0}{\mathcal{VOP}}(\ball{x_0}{r})\supseteq
   \ball{(\nullv,g(x_0))}{\sigma r}.
$$
Since $r<\epsilon$ and $\ball{(\nullv,g(x_0))}{\sigma r}\cap
\mathcal{Q}\ne\varnothing$, the last inclusion entails
$$
    \Image{x_0}{\mathcal{VOP}}(\ball{x_0}{r})\cap\mathcal{Q}
    \ne\varnothing,
$$
what excludes that $x_0$ is $K$-optimal for 
$(\mathcal{VOP}_{x_0,\epsilon})$, according to Proposition
\ref{pro:Koptchar}.
\end{proof}

Now, the analysis is focussed on the subclass of those elements
$x_0\in R$, such that $\Image{x_0}{\mathcal{VOP}}$ is open at
a linear rate near $(x_0,(\nullv,g(x_0)))$. If assuming the data
$h$ and $g$ to be at least strictly differentiable at $x_0$, by
virtue of Theorem \ref{thm:LyuGra} the surjectivity condition on $\FDer
(h,g)(x_0)$ singles out points at which $\Image{x_0}{\mathcal{VOP}}$
is open at a linear rate. These points, while being not solution
to $(\mathcal{VOP})$, nevertheless turn out to enjoy an interesting
property. Indeed, whenever the Polyak's convexity principle can
be invoked, the $\epsilon$-localization of $(\mathcal{VOP})$ around them,
for $\epsilon$ sufficently small, reveal to do admit a solution, which
can be detected by a method proper of convex optimization, i.e.
via an optimality condition stating the maximality of such solution
for the Lagrangian function. Here, by Lagrangian function associated
with $(\mathcal{VOP})$, the classical function $\Lagr:\W^*\times\Y^*
\times\X\longrightarrow\R$, defined by
$$
    \Lagr(w^*,y^*;x)=\langle w^*,h(x)\rangle+\langle y^*,g(x)\rangle,
$$
is meant. This localization property is fomulated in the next result.

\begin{theorem}        \label{thm:locvecopt}
With reference to $(\mathcal{VOP})$, let $\Omega\subseteq\X$ open,
$C\subseteq\Y$ a nonempty, closed and convex set, and $x_0\in\Omega\cap R$.
Suppose that:

\noindent $(i)$ $(\X,\|\cdot\|)$ is uniformly convex with modulus $\delta_\X$
satisfying the quadratic growth condition;

\noindent $(ii)$ $(\W,\|\cdot\|)$ and $(\Y,\|\cdot\|)$ are reflexive Banach spaces;

\noindent $(iii)$ $h,\, g\in\CII(\Omega)$ and $\FDer (h,g)(x_0)\in\Lin(\X,\W\times\Y)$
is onto.

\noindent Then, there exists $\epsilon_0>0$ such that for every $\epsilon\in
(0,\epsilon_0]$ there are $x_\epsilon\in\bd\ball{x_0}{\epsilon}$ and
$(w^*_\epsilon,y^*_\epsilon)\in(\W^*\times\Y^*)\backslash\{(\nullv^*,
\nullv^*)\}$ such that

\begin{eqnarray}
   x_\epsilon \hbox{ is a global solution to } (\mathcal{VOP}_{x_0,\epsilon}); 
\end{eqnarray}
\begin{eqnarray}     \label{in:PPvecopt}
   w^*_\epsilon\in K^\oplus\backslash\{\nullv^*\},\qquad
   -y^*_\epsilon\in\NCone{g(x_\epsilon)}{C};
\end{eqnarray}
\begin{eqnarray}     \label{in:PPvecopt2}
   \Lagr(w^*_\epsilon,y^*_\epsilon;x_\epsilon)\ge
   \Lagr(w^*_\epsilon,y^*_\epsilon;x),\quad\forall
   x\in\ball{x_0}{\epsilon}.
\end{eqnarray}
\end{theorem}

\begin{proof}
Consider the mapping $\Image{\bar x}{\mathcal{VOP}}$ as defined in
$(\ref{eq:defimage})$, with $\bar x=x_0$. Under the assumptions
made, it is $\Image{x_0}{\mathcal{VOP}}\in\CII(\Omega)$ and $\FDer
\Image{x_0}{\mathcal{VOP}}(x_0)=\FDer(h,g)(x_0)$ is onto.
By virtue of hypothesis (i) it is possible to invoke the Polyak's
convexity principle (Theorem \ref{thm:Polyakteo}).
According to it, there exists $\epsilon_0>0$ such that
$\Image{x_0}{\mathcal{VOP}}(\ball{x_0}{\epsilon})$ is a convex
closed set with nonempty interior, for every $\epsilon\in [0,\epsilon_0]$.
Notice that, since $\Image{x_0}{\mathcal{VOP}}$ is continuous
at $x_0$, then without loss of generality one can assume the set
$\Image{x_0}{\mathcal{VOP}}(\ball{x_0}{\epsilon})$ to be also
bounded.
Now, fix any $\epsilon\in (0,\epsilon_0]$. Let us denote by $\Pi_\W:
\W\times\Y\longrightarrow\W$ the projection operator on the space
$\W$. Define $\hat w$ as one of the $K$-minimal element of the set
$$
   \Pi_\W(\Image{x_0}{\mathcal{VOP}}(\ball{x_0}{\epsilon}\cap
  (K\times C)).
$$
Let us show that such definition makes sense, that is
$\hat w$ does actually exist. Since $\FDer
\Image{x_0}{\mathcal{VOP}}(x_0)$ is onto, $x_0$ can not be a
local solution to $(\mathcal{VOP})$. According to Proposition
\ref{pro:Koptchar} it must be
$$
   \Image{x_0}{\mathcal{VOP}}(\ball{x_0}{\epsilon})\cap\mathcal{Q}
   \ne\varnothing
$$
and hence
$$
   \Image{x_0}{\mathcal{VOP}}(\ball{x_0}{\epsilon})\cap (K\times C)
   \ne\varnothing.
$$
Observe that, $\Image{x_0}{\mathcal{VOP}}(\ball{x_0}{\epsilon})$
being convex and closed, it is also weakly closed. As the space
$\W\times\Y$ is reflexive by hypothesis (ii), $\Image{x_0}
{\mathcal{VOP}}(\ball{x_0}{\epsilon})$, which is also bounded,
turns out to be weakly compact.
Since $K\times C$ is convex and closed, it is also weakly closed.
Thus, it is possible to deduce that the nonempty set $(\Image{x_0}
{\mathcal{VOP}}(\ball{x_0}{\epsilon})\cap (K\times C)$ is weakly
compact. On the other hand, the projection mapping $\Pi_\W$
is weakly continuous on $\W\times\Y$. This enables one to conclude
that the image of $\Image{x_0}{\mathcal{VOP}}(\ball{x_0}
{\epsilon}\cap (K\times C)$ through $\Pi_\W$ is weakly compact.
Then, by virtue of a well-known
existence result for vector optimization problems (see, for instance,
Theorem 6.5$(a)$ in \cite{Jahn04}), there must exists $\hat w\in
\Pi_\W(\Image{x_0}{\mathcal{VOP}}(\ball{x_0}{\epsilon})\cap
(K\times C))$, which is $K$-maximal. Now, take any element
$(\hat w,\hat y)\in\Pi_\W^{-1}(\hat w)$. Corresponding to
such a $(\hat w,\hat y)$, there exists $x_\epsilon\in\ball{x_0}
{\epsilon}$ with the property
$$
    \hat w=h(x_\epsilon)-h(x_0),\qquad
    \hat y=g(x_\epsilon).
$$
Notice that, being $(\hat w,\hat y)\in\Image{x_0}{\mathcal{VOP}}
(\ball{x_0}{\epsilon})\cap(K\times C))$, it is $x_\epsilon\in
\ball{x_0}{\epsilon}\cap R$. 
Let us prove that $x_\epsilon$ verifies the first assertion in the
thesis. Ab absurdo, assume that $\tilde x\in\ball{x_0}{\epsilon}
\cap R$ can be found such that
$$
    h(\tilde x)\in (h(x_\epsilon)+K)\backslash\{h(x_\epsilon)\}.
$$
This means that for some $\tilde k\in K\backslash\{\nullv\}$
is has to be
$$
   h(\tilde x)=h(x_\epsilon)+\tilde k,
$$
whence
$$
   h(\tilde x)-h(x_\epsilon)\in K\backslash\{\nullv\}.
$$
Thus, one obtains
\begin{eqnarray}    \label{in:absurdopt}
   h(\tilde x)-h(x_0)=h(\tilde x)-h(x_\epsilon)+h(x_\epsilon)-h(x_0)
  \in\hat w+K,
\end{eqnarray}
with $h(\tilde x)-h(x_0)\ne\hat w$. It follows from $(\ref{in:absurdopt})$
that the $K$-maximality of $\hat w$ is violated. Indeed, it is
$$
   (h(\tilde x)-h(x_0),g(\tilde x))\in\Image{x_0}{\mathcal{VOP}}
  (\ball{x_0}{\epsilon})\cap(K\times C),
$$
because $h(\tilde x)-h(x_0)\in\hat w+K\subseteq K$ and $g(\tilde x)
\in C$.

Observe that the $K$-maximality of $\hat w$ entails that $\hat w\in
\bd\Pi_\W(\Image{x_0}{\mathcal{VOP}}(\ball{x_0}{\epsilon})\cap
(K\times C))$ and this fact, in turn, entails that $(\hat w,\hat y)\in
\bd(\Image{x_0}{\mathcal{VOP}}(\ball{x_0}{\epsilon})\cap
(K\times C))$. Therefore, $(\hat w,\hat y)$ must belong to the boundary
of at least one of the two subsets, $\Image{x_0}{\mathcal{VOP}}
(\ball{x_0}{\epsilon})$ or $K\times C$. If it were $(\hat w,\hat y)
\in\inte\Image{x_0}{\mathcal{VOP}}(\ball{x_0}{\epsilon})$, it would
exist $\hat k\in K\backslash\{\nullv\}$ such that $(\hat w+\hat k,
\hat y)\in\Image{x_0}{\mathcal{VOP}}(\ball{x_0}{\epsilon})$. Since
$(\hat w+\hat k,\hat y)\in K\times C$, this would be inconstistent
with the $K$-maximality of $\hat w$. So one can conclude that
$(\hat w,\hat y)\in\bd\Image{x_0}{\mathcal{VOP}}(\ball{x_0}{\epsilon})$.
Then, according to what noticed in Remark \ref{rem:bdPolyprin},
$x_\epsilon$ must belong to $\bd\ball{x_0}{\epsilon}$. By using
again the characterization of $K$-optimality for problem
$(\mathcal{VOP}_{x_0,\epsilon})$, one obtains
$$
    \Image{x_\epsilon}{\mathcal{VOP}}(\ball{x_0}{\epsilon})\cap
    \mathcal{Q}=\varnothing.
$$
From the definition of mapping $\Image{x_\epsilon}{\mathcal{VOP}}$
one can readily see that it holds
$$
    \Image{x_\epsilon}{\mathcal{VOP}}(x)= \Image{x_0}{\mathcal{VOP}}
   (x)+(h(x_0)-h(x_\epsilon),\nullv),\quad\forall x\in\X.
$$
Therefore, as a mere translation of a convex, closed set with
nonempty interior (remember Remark \ref{rem:bdPolyprin}),
also $\Image{x_\epsilon}{\mathcal{VOP}}(\ball{x_0}
{\epsilon})$ has such properties. Being disjoint from $\mathcal{Q}$,
it can be linearly separated from $\cl\mathcal{Q}=K\times C$, by virtue of the
Eidelheit's theorem. This means that there exists $(w^*_\epsilon,
y^*_\epsilon)\in(\W^*\times\Y^*)\backslash\{(\nullv^*,\nullv^*)\}$ and
$\alpha\in\R$ such that
\begin{eqnarray}     \label{in:Eidel1}
    \langle w^*_\epsilon, h(x)-h(x_\epsilon)\rangle +
    \langle y^*_\epsilon,g(x)\rangle\le\alpha,\quad\forall x\in
    \ball{x_0}{\epsilon},
\end{eqnarray}
and
\begin{eqnarray}     \label{in:Eidel2}
   \langle w^*_\epsilon, w\rangle +\langle y^*_\epsilon,y\rangle
   \ge\alpha,\quad\forall (w,y)\in K\times C.
\end{eqnarray}
If $x=x_\epsilon$, from inequality $(\ref{in:Eidel1})$ one gets
$$
    \langle y^*_\epsilon,g(x_\epsilon)\rangle\le\alpha.
$$
On the other hand, being $(\nullv,g(x_\epsilon))\in K\times C$,
from inequality $(\ref{in:Eidel2})$ one has
$$
    \langle y^*_\epsilon,g(x_\epsilon)\rangle\ge\alpha,
$$
whence
\begin{eqnarray}       \label{eq:compalf}
    \langle y^*_\epsilon,g(x_\epsilon)\rangle=\alpha.
\end{eqnarray}
As for every $y\in C$ it is $(\nullv,y)\in K\times C$, one has
$$
   \langle y^*_\epsilon,y\rangle\ge\alpha,
$$
whence it results in
\begin{eqnarray}      \label{in:Nconecond}
    \langle y^*_\epsilon,y-g(x_\epsilon)\rangle\ge 0,
    \quad\forall y\in C.
\end{eqnarray}
This entails that $-y^*_\epsilon\in\NCone{g(x_\epsilon)}{C}$.
To complete the proof of $(\ref{in:PPvecopt})$, take an arbitrary
$w\in K$. Being $(w,g(x_\epsilon))\in K\times C$, from
$(\ref{in:Eidel2})$ one gets
$$
    \langle w^*_\epsilon,w\rangle\ge 0,
$$
that is $w^*_\epsilon\in K^\oplus$.
To show that $w^*_\epsilon\ne\nullv^*$, assume instead that
$w^*_\epsilon=\nullv^*$. By the open covering property of $g$
around $(x_0,g(x_0))$, which is a consequence of the surjectivity of
$\FDer g(x_0)$, it holds
$$
    g(\ball{x_0}{r})\supseteq\ball{g(x_0)}{\sigma r}
$$
for proper positive $\sigma$ and $r<\epsilon$. In the light of $(\ref{in:Eidel1})$
this yields
$$
    \langle y^*_\epsilon,g(x_0)+\eta u\rangle\le\alpha,\quad
   \forall u\in\Sph,\ \forall\eta\in [0,\sigma r),
$$
whereas, by inequality $(\ref{in:Eidel2})$, the inclusion
 $(\nullv,g(x_\epsilon))\in K\times C$ implies
$$
     \langle y^*_\epsilon,g(x_0)\rangle\ge\alpha.
$$ 
Consequently, one finds
$$
    \eta\langle y^*_\epsilon,u \rangle\le\alpha-
   \langle y^*_\epsilon,g(x_0)\rangle\le 0,\quad\forall
    u\in\Sph,
$$
which evidently contradicts the fact that $y^*_\epsilon\ne\nullv^*$
(remember that $(w^*_\epsilon,y^*_\epsilon)\in\W^*\times\Y^*
\backslash\{(\nullv^*,\nullv^*)\}$).

To conclude the  proof it suffices to observe that $(\ref{in:PPvecopt2})$ 
is a straightforward consequence of inequality $(\ref{in:Eidel1})$ and
of $(\ref{eq:compalf})$.
This completes the proof.
\end{proof}

\begin{remark}     \label{rem:compcon}
(i) In view of a subsequent application of Theorem \ref{thm:locvecopt},
it is to be noted that, whenever set $C$ is, in particular, a cone with
apex at the null vector of $\Y$, then the thesis of the theorem can
be refined by adding that
$$
    y^*_\epsilon\in\{g(x_\epsilon)\}^\perp.
$$
To see this, it suffices to put $y=2g(x_\epsilon)$ and then $y=\nullv$ in
inequality $(\ref{in:Nconecond})$, which is valid all over $C$.

(ii) A remarkable feature of Theorem \ref{thm:locvecopt} is that
the multiplier $w^*_\epsilon$, corresponding to the solution to
$(\mathcal{VOP}_{x_0,\epsilon})$, does not vanish.

(iii) Theorem \ref{thm:locvecopt} extends to the context of vector
optimization Theorem 4.1 in \cite{Poly01a}. Nevertheless, as it is
possible to show by means of easy counterexamples, the uniqueness
of the solution to $\epsilon$-localizations of the problem, which is
valid in scalar optimization, can not be restored in such extension.
\end{remark}

As a comment to Theorem \ref{thm:locvecopt} it is worth noting that
its thesis relates to two different issues arising in the study of
optimization problems. The first one has to do with the existence
of solution to $\epsilon$-localizations of the original problem.
Whereas the solution existence for localizations comes out
automatically in the case of finite-dimensional problems, because
$h$ and $g$ are locally continuous around $x_0$ and $\ball{x_0}
{\epsilon}$ is compact, the question becomes subtler when $\X$
is infinite-dimensional. In such circumstance, indeed, $\ball{x_0}
{\epsilon}$ turns out to be weakly compact as $\X$ is reflexive,
but $h$ may fail to be weakly continuous, in the absence of any 
convexity assumption, and, for a similar reason, $R=g^{-1}(C)$
may fail to be weakly closed. It is at that point that one appreciates the
power of the Polyak's convexity principle.
The second aspect is relevant independently of the dimension
of the underlying space.
It deals with the necessary optimality condition, which turns out to
hold  at a solution to a $\epsilon$-localization of
$(\mathcal{VOP})$. It is well known that standard optimality
conditions for problems with smooth data can only prescribe
stationarity for the Lagrangian function associated with the problem,
in the absence of convexity assumptions. In contrast with this,
resting upon the Polyak's principle, Theorem \ref{thm:locvecopt}
guarantees the maximality of solutions also for the Lagrangian
function, for a proper choice of multipliers.

Again note that,
as it typically happens in convex optimization, conditions 
$(\ref{in:PPvecopt})$ and $(\ref{in:PPvecopt2})$ appearing in
Theorem  \ref{thm:locvecopt} are almost a characterization
of $K$-optimality for problem $(\mathcal{VOP}_{x_0,\epsilon})$.
In other terms, any element of $R\cap\ball{x_0}{\epsilon}$
satisfying condition $(\ref{in:PPvecopt2})$ and an enforcement
of condition $(\ref{in:PPvecopt})$ can be shown to solve
$(\mathcal{VOP}_{x_0,\epsilon})$. This is done below.

\begin{proposition}
Under the same hypotheses of Theorem  \ref{thm:locvecopt},
in the same notations, let $\epsilon\in (0,\epsilon_0]$ and
$z\in\ball{x_0}{\epsilon}$. If there exists $(w^*,y^*)\in(\W^*\times
\Y^*)\backslash\{(\nullv^*,\nullv^*)\}$ fulfilling the following
conditions:
\begin{eqnarray}     \label{in:sufcon1}
   w^*\in K^\oplus,\quad \krn w^*=\{\nullv\},
    \qquad\hbox{and}\qquad
   -y^*\in\NCone{g(z)}{C},
\end{eqnarray}
and
\begin{eqnarray}     \label{in:sufcon2}
   \Lagr(w^*,y^*;z)\ge\Lagr(w^*,y^*;x),\quad\forall
   x\in\ball{x_0}{\epsilon},
\end{eqnarray}
then $z$ is a solution to $(\mathcal{VOP}_{x_0,\epsilon})$.
\end{proposition}

\begin{proof}
Take an arbitrary $x\in R\cap\ball{x_0}{\epsilon}$. Being $g(x)\in C$,
by virtue of the third relation in $(\ref{in:sufcon1})$, one has
$$
    \langle y^*,g(x)-g(z)\rangle\ge 0.
$$
Therefore, from inequality $(\ref{in:sufcon2})$ it follows
$$
    \langle w^*,h(z)\rangle\ge \langle w^*,h(x)\rangle+\langle y^*,
   g(x)-g(z)\rangle\ge\langle w^*,h(x)\rangle,\quad\forall
   x\in R\cap\ball{x_0}{\epsilon},
$$
whence
\begin{eqnarray}    \label{in:whdif}
    \langle w^*,h(z)-h(x)\rangle\ge 0,,\quad\forall
   x\in R\cap\ball{x_0}{\epsilon}.
\end{eqnarray}
Now, assume ab absurdo that $z$ fails to be a solution to
$(\mathcal{VOP}_{x_0,\epsilon})$. Then, there must exists
$\hat x\in R\cap\ball{x_0}{\epsilon}$ such that
\begin{eqnarray}      \label{in:znotopt}
    h(\hat x)\in (h(z)+K)\backslash\{h(z)\},
\end{eqnarray}
and hence
$$
    h(\hat x)-h(z)\in K\backslash\{\nullv\}.
$$
Consequently, since it is $w^*\in K^\oplus$, one finds
$$
    \langle w^*,h(\hat x)-h(z)\rangle\ge 0,
$$
which, along with inequality $(\ref{in:whdif})$, implies
$$
    \langle w^*,h(\hat x)-h(z)\rangle=0.
$$
In the light of the condition in $(\ref{in:sufcon1})$ on the
triviality of $\krn w^*$, the last equality allows one to conclude
that $h(\hat x)=h(z)$, what contradicts inclusion $(\ref{in:znotopt})$.
Thus the proof is complete.
\end{proof}


\section{An application to welfare economics}   \label{Sect:4}

\subsection{The economic model}

This section is concerned with a model of pure exchange economy,
considering finitely many consumers. Private commodities to be consumed
(or desired) by them are formalized as elements of a vector space $(\X,\|\cdot\|)$,
which is assumed to be a real Banach space. This allows one to
modelize economies with an infinite-dimensional commodity space.
Motivations for considering such kind of models, widely recognized
in the modern mathematical economics literature, are discussed
for instance in \cite{AlCoTo02} and in some references therein.

Let $I=\{1,\dots,n\}$ index the consumer set. Each consumer $i\in I$
is described in the model by:
\begin{itemize}

\item a (nonempty) consumption set $\Omega_i\subseteq\X$, representing
the set of commodities, where consumer $i$ makes her choices;

\item a utility function $u_i:\Omega_i\longrightarrow\R$, representing
preferences of the consumer $i$ over commodities.

\end{itemize}
Then, the set
$$
    \Omega=\prod_{i\in I}\Omega_i\subseteq\X^n
$$
defines the {\em social consumption set}, where $\X^n$ is the
$n$ times Cartesian product of $\X$. Its elements are consumption
boundles, denoted by $\commv{x}=(x_1,\dots,x_n)$, with $x_i\in\X$,
for every $i\in I$. $\X^n$ will be structured as a direct sum,
equipped with the $2$-norm $\|\cdot\|_2$. 
Notice that in this model the happiness of each consumer is
affected only by those commodities that she may consume, not by 
commodities considered by other consumers. Such a circumstance
is labelled by stating that the consumers have separable utilities.
As it is classical in general equilibrium theory, this excludes
strategical interactions between consumers (instead typical in
game theory), focussing on how agents in the economy respond
to price systems stimulations.

A vector $\aie\in\X$ denotes the {\em aggregate initial endowment}
of the model, whereas
$\Theta\subseteq\X$ represents the {\em net demand constraint}.
In this setting, a boundle $\commv{x}\in\X^n$ is said to be a
{\em feasible} (or {\em attainable}) {\em allocation} if
$$
    \commv{x}\in\Omega \qquad\hbox{ and }\qquad
    \sum_{i\in I}x_i-\aie\in\Theta.
$$
Notice that, in the case in which $\X$ is a partially ordered vector
space and it is $\Theta=-\X_+$, the feasibility condition for allocations
becomes
$$
    \sum_{i\in I} x_i \le_{\X_+}\aie.
$$
Nevertheless, such an order structure on the commodity space
will be not required in the present approach.

The feasibility constraint is expressed by means of the mapping
$c:\Omega\longrightarrow\X$
$$
    c(\commv{x})=\sum_{i\in I}x_i-\aie.
$$
The set $\mathcal{A}=\{\commv{x}\in\Omega:\  c(x)=\sum_{i\in I}x_i-\aie
\in\Theta\}$ collects all feasible allocations.

The commodity-price duality associated with the model is
indicated by $\langle\X^*,\X\rangle$. This means that the elements
of the dual space $\X^*$ have to be interpreted as prices, so that
the value of a commodity $x\in\X$ at a price $p\in\X^*$ is
denoted by $\langle p,x\rangle$.

The resulting economy is therefore defined by
$$
   \mathcal{E}=(I,\langle\X^*,\X\rangle,(\Omega_i,u_i)_{i\in I},
   \aie,\Theta).
$$
After the pioneering work of L. Walras, given an economy $\mathcal{E}$,
a great amount of quantitative studies on the principles
ruling its mechanism are focussed on general equilibrium theory
(historical commentaries can be found, for instance, in \cite{AlCoTo02,Aubi93,
Mord06II}). In this theory, the concept of Pareto optimal allocation
and the notion of equilibrium play a crucial role.

With reference to an exchange economy $\mathcal{E}$, a feasible
allocation $\bar\commv{x}\in\mathcal{A}$ is said to be {\em Pareto optimal}
if it is $\R^n_+$-optimal for the vector optimization problem
$$
   \Vmax{\R^n_+}_{\commv{x}\in\Omega} u(\commv{x})\quad
   \hbox{ subject to }\quad
   \commv{x}\in\mathcal{A},  \leqno (\mathcal{POP})
$$
where $u:\Omega\longrightarrow\R^n$ is the multiobjective mapping
that arrays the utility functions of all the consumers, i.e.
$$
    u(\commv{x})=(u_1(x_1),\dots, u_n(x_n)).
$$
Notice that, this being the case, the space $\R^n$ is partially ordered
by the natural componentwise order relation. Roughly speaking, Pareto optimality
for $(\mathcal{POP})$ denotes any feasible allocation, which can not
increase the happiness of any consumer without decreasing that
of another one. In this sense, mapping $u$ quantifies the social
efficiency of a given consumption boundle.

The notion of equilibrium is more involved. In the model under
consideration (in fact, in many others), it can not be disjoined
from the notion of supporting price, which lies at the very core
of the decentralization mechanism. A price system, in an equilibrium
situation, should be able to summarize the informations on relative
scarcities in the given economy; consequently, it can be imagined
to induce a distribution among the consumers of the aggregate
initial endowment, according to which each consumer
maximizes her utility function over her budget set, as the latter
results from the endowment distribution. More precisely,
given a price $p\in\X^*\backslash\{\nullv^*\}$, a distribution induced by
$p$ of the aggregate endowment $\aie$ among the consumers is
any boundle $(\aie_i)_{i\in I}\in\X^n$ such that
$$
    \left\langle p,\sum_{i=1}^n \aie_i\right\rangle=\langle p,\aie\rangle.
$$
It is worth noting that endowment distributions are not uniquely
defined by a price system and by the aggregate initial endowment.
The notion of equilibrium can be therefore formalized as follows.

\begin{definition}
With reference to an exchange economy $\mathcal{E}$, a feasible allocation
$\bar\commv{x}\in\mathcal{A}$ is called a ({\em competitive}) {\em equilibrium}
if there exists a price system $\bar p\in\X^*\backslash\{\nullv^*\}$
that supports $\bar\commv{x}$, in the sense that all the following
conditions are fulfilled:

\noindent (i) $\bar p\in\NCone{\sum_{i\in I}\bar x_i-\aie}{\Theta}$
\qquad (price positivity);

\noindent (ii) $\langle\bar p, \sum_{i\in I}\bar x_i\rangle=
\langle\bar p,\aie\rangle$ \qquad (market clear condition);

\noindent (iii) $\bar p$ induces an endowment distribution
$(\aie_i)_{i\in I}$ of $\aie$, according to which
$$
   u_i(\bar x_i)=\max_{x_i\in\Omega_i} u_i(x_i) \hbox{ subject to }
   \langle\bar p,x_i\rangle\le\langle\bar p,\aie_i\rangle,\quad
   \forall i\in I, \qquad \hbox{(individual optimality)}
$$
with $\langle\bar p,\bar x_i\rangle=\langle\bar p,\aie_i\rangle$.
\end{definition}

In the study of welfare economics, the above two notions appear
to be strictly intertwined by two classical fundamental results,
known as first and second welfare theorem. Roughly speaking, under appropriate
assumptions, the first welfare theorem states that every equilibrium
is Pareto optimal, whereas the second one is concerned with the
opposite implication (for their first formulation in a modern
setting the reader is referred to \cite{Arro51} and \cite{Debr54}).
A critical feature of the original theory is that such achievements
can be obtained by making
an essential use of convexity. In the more recent literature devoted
to welfare economics, an active research line revolves around the
extension of the second welfare theorem to models of nonconvex
economies (see, among the others, \cite{AlCoTo02,BonCor88,FlGoJo06,
Gues75,Jofr00,MalMor01,Mord00,Mord06II}). The reason justifying such
an interest has to do with the fact that, as well recognized in the
economic literature, the relevance of convexity assumptions is
doubtful, when even not contradicted in concrete models. A detailed
discussion of such difficulty and of various attempts to overcome it
can be found in the references cited above.

In this paper, starting with the same problem, a different perspective
on the issue is considered. Trying to interpret the spirit of local
programming, it is shown that, even in the very absence of convexity
assumptions, if properly localized, an exchange economy admits
feasible Pareto optimal allocations near a special class of
commodity boundles, here termed regular. If some additional
conditions are satisfied, these Pareto optimal allocations reveal
to be equilibria.


\subsection{Model assumptions}

In this subsection all assumptions, upon which the result next
presented in the paper is established, are listed and discussed.

\begin{itemize}

\item[(${\bf A}_1$)] The commodity space  $(\X,\|\cdot\|)$ is supposed to
be a uniformly convex real Banach space, whose modulus of convexity
satisfies the quadratic growth condition, and such that $(\X^n,\|\cdot\|_2)$
fulfils the same property (remember Remark \ref{rem:UCspace}(iii)).

\item[(${\bf A}_2$)] Each consumption set has nonempty interior, i.e.
$$
    \inte\Omega_i\ne\varnothing,\quad\forall i\in I.
$$
This implies the existence of commodities in the consumption set,
whose small perturbations in any direction still belong to the
consumption set. Technically, such an assumption is connected with
the next one.

\item[(${\bf A}_3$)] The utility function of each consumer is
a $\CII$ function, i.e.
$$
   u_i\in\CII(\inte\Omega_i),\quad\forall i\in I.
$$
As a comment to such assumption, note that, in the model under
consideration, it is implicitly supposed that each consumer's
observed preferences agree with the behaviour axioms, under
which the existence  of an utility function can be derived. The
latter being not a primitive concept, one should complement
the analysis of the behavioural axioms, justifying the specific
property $\CII$ requested on $u_i$. In this regard, take into
account that the assumption on $u_i$ to be ${\rm C}^2$, often
made when dealing with smooth utility functions, entails in
particular (${\bf A}_3$). This stronger assumption is discussed
in \cite{Debr72}.

\item[(${\bf A}_4$)] The net demand constraint set $\Theta$ is a
(nonempty) closed, convex, cone (with apex at $\nullv$). Its
introduction allows one to provide a unifying framework
for different situations arising in economic models. For example,
$\Theta$ may reduce to $\{\nullv\}$, when the market clear condition
is forced by the model over all feasible allocations. Otherwise,
$\Theta$ may coincide with $-\X_+$, in the presence of implicit free
disposal of commodities. Again, it may describe situations in
which information is incomplete or/and uncertainty enters the
economic model.

\item[(${\bf A}_5$)] The following local qualification condition
for the endowment distribution induced by a price system is
supposed to hold: 
for every $\commv{x}\in\mathcal{A}$ and for every $p\in
\NCone{\sum_{i\in I}x_i-\aie}{\Theta}$, $p$ induces an endowment
distribution $(\aie_i)_{i\in I}$ among consumers, such that for
every $\epsilon>0$ and for $i\in I$ there exists $z_i\in\ball{x_i}
{\epsilon}$ with the property
$$
   \langle p,z_i\rangle<\langle p,\aie_i\rangle.
$$
The above condition says that, near the $i$-th component of
each feasible allocation, any price system yields a budget
set for the consumer $i$, which contains, among others, commodities
not exhausting the endowment share distributed to $i$.

\item[(${\bf A}_6$)] Each consumer $i\in I$ is supposed to be
locally non-satiated with respect to subsets of $\Omega_i$.
This amounts to say that, for every $i\in I$, one has
$$
    \forall x\in\Omega_i,\ \forall\epsilon>0, \hbox{ and }
    \forall S\subseteq\Omega_i,\hbox{ with } 
   \ball{x}{\epsilon}\cap S\ne\varnothing, 
    \quad\exists z\in\ball{x}{\epsilon}\cap S
   \hbox{ such that } u_i(z)>u_i(x).
$$
This last assumption is an enforced version of  a well-known
condition, usually appearing in model of welfare economics.

\end{itemize}


\subsection{Regular feasible allocations}

In the setting under consideration, the localization approach
to the analysis of welfare economics leads to single out the following
class of feasible allocations, to which the next result applies.

\begin{definition}
With reference to an exchange economy $\mathcal{E}$, whose elements
satisfy assumptions $({\bf A}_1)-({\bf A}_3)$, a feasible
allocation $\commv{x}_0\in\mathcal{A}$ is said to be {\em regular}
if $\commv{x}_{0,i}\in\inte\Omega_i$, for every $i\in I$, and
$\FDer(u,c)(\commv{x}_0)$ is onto.
\end{definition}

\begin{remark}
As an immediate consequence of Proposition \ref{pro:openvectopt},
one has that if $\commv{x}_0\in\mathcal{A}$ is a regular allocation
for $\mathcal{E}$, then it can not be a Pareto optimal one for any
$\epsilon$-localization around $\commv{x}_0$ of problem
$(\mathcal{POP})$.
\end{remark}

When dealing with equilibria of an economy, the notion of problem
localization must be somehow adapted. Namely,
given a feasible allocation $\commv{x}_0$ and $n$ positive reals
$\epsilon_1,\dots,\epsilon_n$, a $(\epsilon_1,\dots,\epsilon_n)$-localization
of an economy $\mathcal{E}$ around $\commv{x}_0$ is the exchange
economy defined by
\begin{eqnarray}      \label{eq:localecon}
   \mathcal{E}_{\commv{x}_0,\epsilon_1,\dots,\epsilon_n}=
   (I,\langle\X,\X^*\rangle,(\ball{x_{0,i}}{\epsilon_i},u_i)_{i\in I},
   \aie,\Theta).
\end{eqnarray}
Having done that, one is in a position to formulate the following
result.

\begin{theorem}      \label{thm:welecon}
Let $\mathcal{E}$ be an exchange economy satisfying assumptions
$({\bf A}_1)-({\bf A}_4)$ and let $\commv{x}_0\in\mathcal{A}$ be a
regular feasible allocation for $\mathcal{E}$. Then there exists
$\epsilon_0>0$ such that, for every $\epsilon\in (0,\epsilon_0]$
there is $\commv{x}_\epsilon\in\mathcal{A}\cap\bd\ball{\commv{x_0}}
{\epsilon}$, which is Pareto optimal for $(\mathcal{POP}_{\commv{x_0},
\epsilon})$.
If, in addition, assumption $({\bf A}_5)-({\bf A}_6)$ hold true and,
letting $\eta_i=\|x_{\epsilon,i}-x_{0,i}\|$, it results in
$$
    \min\{\eta_1,\dots,\eta_n\}>0,
$$
then
such a $\commv{x}_\epsilon$ is an equilibrium of the localized economy
$\mathcal{E}_{\commv{x_0},\eta_1,\dots,\eta_n}$.
\end{theorem}

\begin{proof}
Since $\commv{x}_0$ is regular, it belongs to $\prod_{i\in I}\inte\Omega_i$
and $\FDer(u,c)(\commv{x}_0)$ is onto. Thus, under the assumptions 
$({\bf A}_1)-({\bf A}_4)$, $\R^n$ and $\X$ being reflexive spaces, it is possible
to apply Theorem \ref{thm:locvecopt}, with $h$, $g$, $\Omega$ and
$C$ replaced by $u$, $c$,  $\prod_{i\in I}\inte\Omega_i$ and $\Theta$,
respectively. As a consequence, one gets the existence of $\epsilon_0
>0$ such that, for every $\epsilon\in (0,\epsilon_0]$, there exists
$\commv{x}_\epsilon\in\bd\ball{\commv{x}_0}{\epsilon}\cap\mathcal{A}$
solving the localized problem
$$
   \Vmax{\R^n_+}_{\commv{x}\in\ball{\commv{x}_0}{\epsilon}} u(\commv{x})
   \quad\hbox{ subject to }\quad\commv{x}\in\mathcal{A}. 
$$
The reader should notice that by $({\bf A}_2)$, without loss of
generality, up to a reduction of $\epsilon_0$, one can
assume that $\ball{\commv{x}_0}{\epsilon_0}\subseteq\prod_{i\in I}
\Omega_i$, and hence $\ball{x_{0,i}}{\epsilon_0}\subseteq\Omega_i$,
for every $i\in I$. This prove the first assertion in the theorem.
As to the second one, fix $\epsilon\in (0,\epsilon_0]$ and
note that, along with the existence of
$\commv{x}_\epsilon$, Theorem  \ref{thm:locvecopt}
enables one to get the existence of $(\mu_\epsilon,x^*_\epsilon)\in
(\R^n\times\X^*)\backslash\{(\nullv,\nullv^*)\}$, such that
\begin{eqnarray}      \label{in:sep1}
     \mu_\epsilon\in\R^n_+\backslash\{\nullv\},\qquad
     -x^*_\epsilon\in\NCone{\sum_{i\in I}x_{\epsilon,i}-\aie}
     {\Theta},
\end{eqnarray}
and
\begin{eqnarray}       \label{in:sep2}
     \sum_{i\in I}\mu_{\epsilon,i}u_i(x_{\epsilon,i})+
   \langle x^*_\epsilon,\sum_{i\in I}x_{\epsilon,i}-\aie\rangle
\ge
    \sum_{i\in I}\mu_{\epsilon,i}u_i(x_i)+\langle x^*_\epsilon,
    \sum_{i\in I}x_i-\aie\rangle,\quad\forall\commv{x}\in
     \ball{\commv{x}_0}{\epsilon}.
\end{eqnarray}
Set $p_\epsilon=-x^*_\epsilon$, let us check that a multiple of $p_\epsilon$
supports the allocation $\commv{x}_\epsilon$, with reference to
the localized economy $\mathcal{E}_{\commv{x_0},\eta_1,\dots,\eta_n}$.
According to the position of $\eta_i$, it is $\commv{x}_\epsilon\in
\prod_{i=1}^n\ball{\commv{x}_0}{\eta_i}$, so $\commv{x}_\epsilon$
is a feasible allocation for $\mathcal{E}_{\commv{x_0},\eta_1,
\dots,\eta_n}$. The positivity
of the price $p_\epsilon$ is expressed by the second inclusion in
$(\ref{in:sep1})$. Since in $({\bf A}_4)$ $\Theta$ has been assumed
to be a cone, as noted in Remark \ref{rem:compcon}(i), one has also
\begin{eqnarray}      \label{eq:pepscmc}
     \langle p_\epsilon,\sum_{i\in I}x_{\epsilon,i}-\aie\rangle=0,
\end{eqnarray} 
which is exactly the market clear condition. To show that actually
$p_\epsilon\ne\nullv^*$, suppose to the contrary that $p_\epsilon$
vanishes. From inequality $(\ref{in:sep2})$, it follows
\begin{eqnarray}     \label{in:abspeps0}
    \sum_{i\in I}\mu_{\epsilon,i}u_i(x_{\epsilon,i})\ge
    \sum_{i\in I}\mu_{\epsilon,i}u_i(x_i),\quad\forall
    \commv{x}\in\ball{\commv{x}_0}{\epsilon}.
\end{eqnarray}
Since it is $\mu_\epsilon\ne\nullv$, a proper $j\in I$ can be found
such that $\mu_{\epsilon,j}>0$. By virtue of $({\bf A}_6)$,
taking $S=\ball{x_{0,j}}{\eta_j}$,
there exists $z_j\in\ball{x_{\epsilon,j}}{\epsilon}\cap\ball{x_{0,j}}
{\eta_j}$, with the property
$$
    u_j(x_{\epsilon,j})<u_j(z_j).
$$
Thus, if taking the boundle $\tilde\commv{x}$ defined by
\begin{eqnarray*}
     \tilde x_i=\left\{\begin{array}{ll}
         x_{\epsilon,i}, & \quad\forall i\in I\backslash\{j\}, \\
         z_j, & \quad\hbox{if } i=j,        
      \end{array}\right.
\end{eqnarray*}
it is $\tilde\commv{x}\in\ball{\commv{x}_0}{\epsilon}$. Indeed,
it results in
\begin{eqnarray*}
  \|\tilde\commv{x}-\commv{x}_0\| &=& \left( \sum_{i\in I\backslash\{j\}}
  \|x_{\epsilon,i}-x_{0,i}\|^2+\|z_j-x_{0,j}\|^2 \right)^{1/2}=
  \left(\epsilon^2- \|x_{\epsilon,j}-x_{0,j}\|^2+ \|z_j-x_{0,j}\|^2
  \right)^{1/2}  \\
   &=& \left(\epsilon^2-\eta_j^2+ \|z_j-x_{0,j}\|^2\right)^{1/2}\le\epsilon.
\end{eqnarray*}
Thus, being
$$
    \sum_{i\in I}\mu_{\epsilon,i}u_i(x_{\epsilon,i})<
    \sum_{i\in I}\mu_{\epsilon,i}u_i(\tilde x_i),
$$
one obtains a contradiction of $(\ref{in:abspeps0})$. Therefore,
$p_\epsilon\ne\nullv^*$. In order to complete the proof, it
remains to check the individual
optimality condition. To this aim, let us consider a distribution
$(\aie_{\epsilon,i})_{i\in I}$ of the aggregate endowment $\aie$
among the consumers, which is induced by $p_\epsilon$, in such a
way that
\begin{eqnarray}      \label{eq:enddis}
    \langle p_\epsilon,\aie_{\epsilon,i}\rangle=
    \langle p_\epsilon,x_{\epsilon,i}\rangle,
   \quad\forall i\in I.
\end{eqnarray}
This can be done, because it is
$$
    \left\langle p_\epsilon,\sum_{i\in I}\aie_{\epsilon,i}
    \right\rangle=\sum_{i\in I} \langle p_\epsilon,
    x_{\epsilon,i}\rangle=\langle p_\epsilon,
   \aie\rangle,
$$
according to $(\ref{eq:pepscmc})$.
Now, fix an arbitrary $j\in I$. For every $z_j\in\ball{x_{0,j}}
{\eta_j}$, any boundle $\commv{x}$ of the form
\begin{eqnarray*}
      x_i=\left\{\begin{array}{ll}
         x_{\epsilon,i} & \quad\forall i\in I\backslash\{j\}, \\
         z_j & \quad\hbox{if } i=j      
      \end{array}\right.
\end{eqnarray*}
still belongs to $\ball{\commv{x}_0}{\epsilon}$, as already observed.
Consequently,
from $(\ref{in:sep2})$, taking $(\ref{eq:pepscmc})$ into account,
one obtains
\begin{eqnarray*}    
    \mu_{\epsilon,j}u_j(x_{\epsilon,j}) &\ge&   
    \mu_{\epsilon,j}u_j(z_j)-\langle p_\epsilon,
    \sum_{i\in I\backslash\{j\}}x_{\epsilon,i}+z_j-\aie\rangle \\
    &=&  \mu_{\epsilon,j}u_j(z_j)-\langle p_\epsilon,
    z_j-x_{\epsilon,j}\rangle,
   \quad\forall  z_j\in\ball{x_{0,j}}{\eta_j}.
\end{eqnarray*}
From the last inequality, in force of the endowment distribution
$(\ref{eq:enddis})$, it follows
\begin{eqnarray}       \label{in:dectrdomin}
     \mu_{\epsilon,j}u_j(x_{\epsilon,j})\ge\mu_{\epsilon,j}u_j(z_j)
   -\langle p_\epsilon, z_j-\aie_{\epsilon,j}\rangle,
   \quad\forall  z_j\in\ball{x_{0,j}}{\epsilon}.
\end{eqnarray}
Notice that $\mu_{\epsilon,j}$ must be positive. Otherwise, it
would be
$$
    \langle p_\epsilon, z_j-\aie_{\epsilon,j}\rangle\ge 0,
$$
so the arbitrariness of $z_j$ would contradict the existence of
$\hat z_j\in\ball{x_{0,j}}{\epsilon}\subseteq\Omega_j$
such that
$$
   \langle p_\epsilon,\hat z_j\rangle<\langle p_\epsilon,
   \aie_{\epsilon,j}\rangle,
$$
which is guaranteed by virtue of the local qualification condition
(${\bf A}_5$). Thus, the positivity of $\mu_{\epsilon,j}$
enables one to obtain from $(\ref{in:dectrdomin})$
$$
    u_j(x_{\epsilon,j})\ge u_j(z_j),\quad\forall z_j\in
   \ball{x_{0,j}}{\eta_j}\hbox{ such that } 
   \left\langle {p_\epsilon\over\mu_{\epsilon,j}},z_j\right\rangle\le
   \left\langle {p_\epsilon\over\mu_{\epsilon,j}},\aie_{\epsilon,j}
   \right\rangle,
$$
what means that $x_{\epsilon,j}$ is optimal for the consumer $j\in I$
over her budget set.
Since the positivity condition and the market clear condition
are both invariant with respect to the multiplication by positive scalars,
one can take as a price system supporting $\commv{x}_\epsilon$
the functional $\bar p_\epsilon=p_\epsilon/\mu_{\epsilon,j}$.
This completes the proof.
\end{proof}

It should be clear that Theorem \ref{thm:welecon} is not a generalization
of the second welfare theorem to an exchange economy model affected
by non-convexities. What it states is rather different. First of all
it is an existence result. More precisely, it speaks about the
behaviour of an exchange economy near its regular feasible allocations,
provided that this economy is localized as in $(\ref{eq:localecon})$.
Of course, any such localization modifies the geometry of the problem,
with the result of yielding the existence of Pareto optimal allocations.
The second part of the thesis, which is valid under additional
assumptions, qualifies the above Pareto optimal allocations as
equilibria.

As in many recent generalizations of the second welfare theorem,
also in Theorem \ref{thm:welecon} some convexity assumption is
dropped out: in fact, utility functions are merely supposed to be $\CII$.
The convexity of the net demand constraint set is maintained
(in the cases $\Theta=\{\nullv\}$ and $\Theta=-\X_+$, such assumption is
automatically satisfied), because the price system supporting
an equlibrium is still obtained by means of the linear separation
theorem. This is evidently in contrast with many of the aforementioned
generalized second welfare theorems, which rely on a nonconvex
separation technique (see \cite{FlGoJo06,Jofr00}) due to
J.M. Borwein and A. Jofr\'e or on the so-called extremal principle
due to B.S. Mordukhovich (see \cite{MalMor01,Mord00,Mord06II}).


\end{document}